\documentclass[11pt,english]{article}

\usepackage[margin = 2.4cm]{geometry}

\usepackage{amsthm}
\usepackage{amsmath}
\usepackage{amssymb}
\usepackage{mathtools}
\usepackage{graphicx}
\usepackage[hidelinks, bookmarks = false, colorlinks]{hyperref}
\usepackage{enumitem}

\usepackage{caption}
\usepackage[square,sort,comma,numbers]{natbib} 
\usepackage{setspace}
\usepackage{cleveref}
\theoremstyle{plain}

\newtheorem*{thm*}{Theorem}
\newtheorem{thm}{Theorem}
\crefname{thm}{Theorem}{Theorems}
\Crefname{thm}{Theorem}{Theorems}

\newtheorem*{lem*}{Lemma}
\newtheorem{lem}[thm]{Lemma}
\crefname{lem}{Lemma}{Lemmas}
\Crefname{lem}{Lemma}{Lemmas}

\newtheorem*{claim*}{Claim}

\crefname{claim}{Claim}{Claims}
\Crefname{claim}{Claim}{Claims}

\crefname{prop}{Proposition}{Propositions}
\Crefname{prop}{Proposition}{Propositions}

\crefname{cor}{Corollary}{Corollaries}
\Crefname{cor}{Corollary}{Corollaries}

\crefname{conj}{Conjecture}{Conjectures}
\Crefname{conj}{Conjecture}{Conjectures}

\newtheorem{qn}[thm]{Question}
\crefname{qn}{Question}{Questions}
\Crefname{qn}{Question}{Questions}

\crefname{obs}{Observation}{Observations}
\Crefname{obs}{Observation}{Observations}

\crefname{ex}{Example}{Examples}
\Crefname{ex}{Example}{Examples}

\theoremstyle{definition}

\crefname{prob}{Problem}{Problems}
\Crefname{prob}{Problem}{Problems}

\crefname{defn}{Definition}{Definitions}
\Crefname{defn}{Definition}{Definitions}

\theoremstyle{remark}

\crefname{rem}{Remark}{Remarks}
\Crefname{rem}{Remark}{Remarks}

\usepackage{xpatch}
\xpatchcmd{\proof}{\itshape}{\normalfont\proofnamefont}{}{}
\newcommand{\proofnamefont}{}
\renewcommand{\proofnamefont}{\bfseries}

\usepackage{framed}

\newcommand{\remove}[1]{}

\newcommand{\floor}[1]{
    \left\lfloor #1 \right\rfloor
}

\DeclareMathOperator{\coords}{co}
\DeclareMathOperator{\dens}{dens}
\DeclareMathOperator{\exx}{ex}
	
\newcommand{\HH}{\mathcal{H}}
\newcommand{\C}{\mathcal{C}}

\newcommand{\eps}{\varepsilon}
\newcommand{\lam}{\lambda}

\newcommand{\excr}{\exx_r(n, \C)}

\title{Hypergraphs with no tight cycles}

\author{
	    Shoham Letzter\thanks{
		Department of Mathematics, 
		University College London, 
		Gower Street, London WC1E~6BT, UK. 
		Email: \texttt{s.letzter}@\texttt{ucl.ac.uk}. 
		Research supported by the Royal Society.
    }
}

\begin{document}

\date{}
\maketitle

\begin{abstract}

	\setlength{\parskip}{\medskipamount}
    \setlength{\parindent}{0pt}
    \noindent

	We show that every $r$-uniform hypergraph on $n$ vertices which does not contain a tight cycle has at most $O(n^{r-1} (\log n)^5)$ edges. 
	This is an improvement on the previously best-known bound, of $n^{r-1} e^{O(\sqrt{\log n})}$, due to Sudakov and Tomon, and our proof builds on their work. A recent construction of B.\ Janzer implies that our bound is tight up to an $O((\log n)^4 \log \log n)$ factor.

\end{abstract}

\section{Introduction} \label{sec:intro}

	It is well known, and easy to see, that the maximum number of edges in a graph on $n$ vertices with no cycles is $n-1$. It is natural to consider an analogous problem for hypergraphs: what is the maximum possible number of edges in an $r$-uniform hypergraph (henceforth \emph{$r$-graph}) on $n$ vertices which does not contain a cycle? Unlike the graph case, there are multiple natural notions of cycles in hypergraphs, the most notable of which are Berge cycles, loose cycles and tight cycles. 

	A \emph{Berge cycle} of length $\ell$ is a sequence $(v_1, e_1, \ldots, v_{\ell}, e_{\ell})$ such that $v_1, \ldots, v_{\ell}$ are distinct vertices, $e_1, \ldots, e_{\ell}$ are distinct edges, and $v_i \in e_{i-1} \cap e_i$ (subtraction of indices is taken modulo $\ell$). We claim that the maximum possible number of edges in an $n$-vertex $r$-graph with no Berge cycles is $\floor{\frac{n-1}{r-1}}$. For the upper bound, it suffices to show that the edges of an $r$-graph with no Berge cycles can be ordered as $e_1, \ldots, e_m$ so that $|e_i \cap (e_1 \cup \ldots \cup e_{i-1})| \le 1$ for every $i \in [m]$, which is not hard to prove. To see the lower bound, form an $r$-graph on at most $n$ vertices by taking $\floor{\frac{n-1}{r-1}}$ pairwise disjoint sets of size $r-1$, and joining each of them to the same new vertex.

	A \emph{loose cycle} of length $\ell$ is a sequence $(e_1, \ldots, e_{\ell})$ of distinct edges such that two consecutive edges (as well as the first and last) have exactly one vertex in common, and non-consecutive edges are disjoint. Frankl and F\"uredi \cite{frankl1987exact} showed that any $n$-vertex $r$-graph with no loose triangles (i.e.\ loose cycles of length $3$) has at most $\binom{n-1}{r-1}$ edges, whenever $n$ is sufficiently large. Note that there exists an $n$-vertex $r$-graph with no loose cycles with this number of edges: take its edges to be all $r$-sets that contain a certain vertex $u$. It thus follows that the answer to the above question for loose cycles is $\binom{n-1}{r-1}$. 

	An $r$-uniform \emph{tight cycle} of length $\ell$ is a sequence $(v_1, \ldots, v_{\ell})$ of distinct vertices, satisfying that $(v_i, \ldots, v_{i+r-1})$ is an edge for every $i \in [\ell]$ (with addition of indices taken modulo $\ell$). Denote the family of all tight cycles by $\C$, and let $\excr$ be the maximum possible number of edges in an $n$-vertex $r$-graph with no tight cycles. The question from the first paragraph, for tight cycles, can be restated as follows: what is $\excr$?

	It might be tempting to guess that $\excr = \binom{n-1}{r-1}$, similarly to the loose cycles case. Indeed, this was conjectured by S\'os and, independently, Verstra\"ete (see \cite{verstraete2016extremal,mubayi2011hypergraph}). This conjecture was disproved by Huang and Ma \cite{huang2019tight}, who showed that for every $r$ there exists $c = c(r) \in (1, 2)$ such that $\excr \ge c \cdot \binom{n-1}{r-1}$. Very recently, B.\ Janzer improved this lower bound on $\excr$ substantially, showing that $\excr = \Omega(n^{r-1} \cdot \frac{\log n}{\log \log n})$.

	Until recently, the best upper bound on $\excr$ for general $r$ was $\excr = O(n^{r - 2^{-(r-1)}})$, which follows from a result of Erd\H{o}s \cite{erdos1964extremal} about the extremal number of a complete $r$-partite $r$-graph with vertex classes of size $2$. For $r = 3$, an unpublished result of Verstra\"ete regarding the extremal number of a tight cycle of length $24$ implies that $\exx_3(n, \C) = O(n^{5/2})$. A recent result of Tomon and Sudakov \cite{sudakov2020extremal} shows that $\excr \le n^{r-1} e^{O(\sqrt{\log n})}$, greatly improving on previous bounds, and thus establishing that $\excr = n^{r-1 + o(1)}$. 

	We prove the following result about the extremal number of tight cycles in $r$-graphs, which lowers the $e^{O(\sqrt{\log n})}$ error term in Sudakov and Tomon's bound to a polylogarithmic term.
		
	\begin{thm} \label{thm:main}
		Suppose that $\HH$ is an $r$-graph on $n$ vertices which does not contain a tight cycle. Then $\HH$ has $O(n^{r-1} (\log n)^5)$ edges.
	\end{thm}

	In other words, we show that $\excr = O(n^{r-1} (\log n)^5)$. In light of Janzer's result \cite{janzer2020large}, this is tight up to an $O((\log n)^4 \log \log n)$ factor. 

	We give an overview of our proof in \Cref{sec:overview}, mention relevant tools and definitions from \cite{sudakov2020extremal} in \Cref{sec:prelims}, and prove our main result in \Cref{sec:cycles-in-expanders}. We conclude the paper in \Cref{sec:conclusion} with some closing remarks. Throughout the paper, logarithms are understood to be in base $2$, and floor and ceiling signs are often dropped.

\section{Overview of the proof} \label{sec:overview}

	Our proof builds on ideas Sudakov and Tomon's work \cite{sudakov2020extremal}. They introduce the notions of $r$-line-graphs, which are graphs that correspond naturally to $r$-partite $r$-graphs, and expansion in such graphs. They show that, given a dense enough $r$-partite $r$-graph $\HH$, the $r$-line-graph that corresponds to $\HH$ contains a dense expander $G$. Next, they define $\sigma$-paths and $\sigma$-cycles, which correspond to tight paths and cycles in the original hypergraph $\HH$. It thus suffices to show that every $r$-line-graph which is a dense expander contains a $\sigma$-cycle. Sudakov and Tomon are not able to prove this. Instead, they show that every expander contains either a $\sigma$-cycle or a very dense subgraph, and proceed via a density increment argument. 

	Our main contribution is to show that every $r$-line-graph which is a dense expander indeed contains a $\sigma$-cycle (see \Cref{thm:cycles-in-expanders}). A key step in our proof is to show that in such an expander $G$, for every vertex $x \in V(G)$, almost every other vertex $y \in V(G)$ can be reached from $x$ via a short $\sigma$-path $P(x, y)$ in a `robust' way, meaning that no vertex in the underlying $r$-graph $\HH$ meets too many of the paths $P(x, y)$ (see \Cref{lem:robust-expanding}). If the robustness requirement is dropped, we obtain a lemma from \cite{sudakov2020extremal}. To prove the robust version, we use the non-robust version from \cite{sudakov2020extremal} as a black box, along with another lemma from the same paper, which asserts that the removal of a small number of vertices from the underlying $r$-graph $\HH$ does not ruin the expansion. 

	To find a $\sigma$-cycle, let $P(x, y)$ be paths as above, defined for almost every $x, y \in V(G)$. Note that while we are guaranteed that, for every $x \in V(G)$, no vertex $v$ of $\HH$ meets too many paths $P(x, y)$, we do not have any control over the number of times $v$ meets a path $P(x, y)$, for a given $y$. Nevertheless, since the paths $P(x, y)$ are short, for every $y \in V(G)$ there are few vertices in $\HH$ that meet many path $P(x, y)$; denote the set of such vertices in $\HH$ by $F(y)$. Using tools mentioned above, for every $y$ and almost every $x$ there is a short $\sigma$-path $Q(y, x)$ from $y$ to $x$ that avoids $F(y)$. To complete the proof, we note that the robustness implies that for almost every $x, y \in V(G)$ the path $Q(y, x)$ is defined, and there are linearly many $z \in V(G)$ for which $P(x, z) P(z, y)$ is a $\sigma$-path from $x$ to $y$. Using robustness and the choice of $Q(y, x)$, the concatenation $P(x, z) P(z, y) Q(y, x)$ is a $\sigma$-cycle for linearly many $z \in V(G)$. 

\section{Expansion in $r$-line-graphs} \label{sec:prelims}

	We say that $G$ is an \emph{$r$-line-graph} if the vertex set of $G$ is a set of $r$-tuples in $A_1 \times \ldots \times A_r$, where $A_1, \ldots, A_r$ are pairwise disjoint, and $x$ and $y$ are joined by an edge if and only if $x$ and $y$ differ in exactly one coordinate. Observe that an $r$-partite $r$-graph naturally corresponds to an $r$-line-graph. 

	Let $G$ be an $r$-line-graph with $V(G) \subseteq A_1 \times \ldots \times A_r$.
	We will refer to the vertices of $A_1 \cup \ldots \cup A_r$ as \emph{coordinates}. For a set of vertices $X$ in $G$, let $\coords(X)$ be the set of coordinates that appear in tuples in $X$. For a vertex $x$ we write $\coords(x)$ as a shorthand for $\coords(\{x\})$.

	For a vertex $x$ and $i \in [r]$, define $N^{(i)}(x)$ to be the set of vertices $y$ in $G$ that differ from $x$ in the $i$-th coordinate only. An \emph{$i$-block} in $G$ is a set of form $\{x\} \cup N^{(i)}(x)$, for $x \in V(G)$ and $i \in [r]$. Let $p(G)$ be the number of blocks in $G$, and define the \emph{density} of $G$, denoted $\dens(G)$, as 
	\begin{equation} \label{eqn:density}
		\dens(G) = \frac{\sum_B|B|}{p(G)} = \frac{r|G|}{p(G)},
	\end{equation}
	where the sum is over all blocks $B$ in $G$. In words, the density is the average size of a block.

	The \emph{$i$-degree} of a vertex $x$, denoted $d_G^{(i)}(x)$, is defined to be $|N^{(i)}(x)| + 1$. The \emph{minimum degree} of $G$, denoted $\delta(G)$, is defined to be the minimum of $d^{(i)}(x)$, over $x \in V(G)$ and $i \in [r]$ (this is not quite the same as the usual notion of a minimum degree of a graph).

	For a graph $H$, say that $H$ is a \emph{$\lam$-expander} if every set of vertices $X$ with $|X| \le \frac{1}{2}|H|$ satisfies $|N(X)| \ge \lam |X|$, where $N(X)$ is the set of vertices in $V(H) \setminus X$ that are neighbours of at least one vertex in $X$. For an $r$-line-graph $G$, say that $G$ is a \emph{$(\lam, d)$-expander} if $G$ is a $\lam$-expander and $\delta(G) \ge d$. 

	The following lemma from \cite{sudakov2020extremal} allows us to find expanders in $r$-line-graphs that are sufficiently dense. It is reminiscent of a similar result of Shapira and Sudakov \cite{shapira2015small} about the existences of expanders in graphs.

	\begin{lem}[Lemma 3.2 in \cite{sudakov2020extremal}] \label{lem:existence-expanders}
		Let $G$ be an $r$-line-graph on $n$ vertices with density at least $d$, and suppose that $0 < \lam \le \frac{1}{2\log n}$. Then $G$ contains a subgrah of density at least $d(1 - \lam \log n)$ which is a $(\lam, \frac{d}{2r})$-expander.
	\end{lem}

	The following lemma, also from \cite{sudakov2020extremal}, shows that the notion of expansion is robust, in the sense that the removal of a small number of coordinates does not affect the expansion too much.

	\begin{lem}[Lemma 3.3 in \cite{sudakov2020extremal}] \label{lem:robust-expander}
		Let $r, u, d$ be positive integers, let $\lam \in (0, 1)$ and suppose that $u \le \frac{\lam d}{4r}$. Let $G$ be an $r$-line-graph on $n$ vertices with $V(G) \subseteq A_1 \times \ldots \times A_r$ which is a $(\lam, d)$-expander. Suppose that $H$ is a subgraph of $G$ obtained by removing at most $u$ coordinates in $A_1 \cup \ldots \cup A_r$ from $G$ (along with edges of $G$ that meet these coordinates). Then $H$ is an $r$-line graph on at least $(1 - \frac{u}{d})n$ vertices which is a $(\frac{\lam}{2}, \frac{d}{2})$-expander.
	\end{lem}

	Next, we need the notions of $\sigma$-neighbours, $\sigma$-paths and $\sigma$-cycles. Let $G$ be an $r$-line-graph with $V(G) \subseteq A_1 \times \ldots \times A_r$. Given a permutation $\sigma \in S_r$ and vertices $x = (x_1, \ldots, x_r)$ and $y = (y_1, \ldots, y_r)$ in $G$, we say that $y$ is a \emph{$\sigma$-neighbour} of $x$ if $\coords(x)$ and $\coords(y)$ are disjoint, and the $r$-tuples $z_0, \ldots, z_{r}$, defined as follows, are vertices in $G$.
	\begin{equation*}
		(z_i)_{j} = \left\{
			\begin{array}{ll}
				x_{j} & \sigma^{-1}(j) > i \\
				y_{j} & \sigma^{-1}(j) \le i.
			\end{array}
		\right.
	\end{equation*}
	Note that $z_0 = x$ and $z_r = y$. If $\sigma$ is the identity permutation, we have $z_i = (y_1, \ldots, y_{i-1}, x_i, \ldots, x_r)$. Observe that $z_{i} \in N^{(\sigma(i))}(z_{i-1})$ for $i \in [r]$. Also note that if $y$ is a $\sigma$-neighbour of $x$ then the sequence $(x_{\sigma(1)}, \ldots, x_{\sigma(r)}, y_{\sigma(1)}, \ldots, y_{\sigma(r)})$ is a tight path in the $r$-graph that corresponds to $G$.\footnote{For the purpose of this paper it suffices to fix $\sigma$ to be any particular permutation in $S_r$. We state the definitions and results for general $\sigma$ to mirror \cite{sudakov2020extremal}.}

	A \emph{$\sigma$}-path in $G$ is a sequence $(x_1, \ldots, x_k)$ of vertices in $G$ whose coordinate sets are pairwise disjoint, and such that $x_{i+1}$ is a $\sigma$-neighbour of $x_{i}$ for $i \in [k-1]$. Similarly, a \emph{$\sigma$-cycle} is a sequence $(x_1, \ldots, x_k)$ of vertices in $G$ whose coordinate sets are pairwise disjoint, such that $x_{i+1}$ is a $\sigma$-neighbour of $x_i$, for $i \in [k]$ (with indices taken modulo $k$). Writing $x_i = (x_{i,1}, \ldots, x_{i,r})$, if $x_1, \ldots, x_r$ is a $\sigma$-path ($\sigma$-cycle), then $(x_{1, \sigma(1)}, \ldots, x_{1, \sigma(r)}, \ldots, x_{k, \sigma(1)}, \ldots, x_{k, \sigma(r)})$ is a tight path (cycle) in the $r$-graph corresponding to $G$. It would thus be useful to show that $r$-line-graphs that are dense expanders have $\sigma$-cycles; we do so in \Cref{thm:cycles-in-expanders} below. 

	The \emph{order} of a $\sigma$-path or $\sigma$-cycle $(x_1, \ldots, x_k)$ is $k$. If there is a $\sigma$-path $(x_1, \ldots, x_k)$ in $G$, we say that \emph{$x_k$ can be reached from $x_1$} by a $\sigma$-path of order $k$. The following lemma from \cite{sudakov2020extremal} shows that, given a vertex $x$ in an $r$-line-graph $G$ which is a dense expander, almost every vertex in $G$ can be reached from $x$ by a relatively short $\sigma$-path.

	\begin{lem}[Lemma 4.2 in \cite{sudakov2020extremal}] \label{lem:expanding}
		Let $\sigma \in S_r$, let $\eps, \lam \in (0,1)$ and let $n$ and $d$ be positive integers such that $500r^4 \log n < \eps^2 \lam^2 d$. Suppose that $G$ is an $r$-line-graph on $n$ vertices which is a $(\lam, d)$-expander, and let $x \in V(G)$. Then at least $(1 - \eps)n$ vertices in $G$ can be reached from $x$ by a $\sigma$-path of order at most $\frac{5r\log n}{\eps \lam}$.
	\end{lem}

\section{Existence of $\sigma$-cycles in expanders} \label{sec:cycles-in-expanders}

	Recall that $\coords(X)$, where $X$ is a set of vertices in an $r$-line-graph, is the set of coordinates in tuples in $X$. The following key lemma is the first new ingredient in our proof. It shows that for every vertex $x$ in an $r$-line-graph $G$ which is a dense expander, almost every vertex in $G$ can be reached from $x$ by a short $\sigma$-path, such that no coordinate (other than the coordinates in $x$) is met by too many such $\sigma$-paths.

	\begin{lem} \label{lem:robust-expanding}
		Let $\sigma \in S_r$, let $\eps, \lam \in (0,1)$ and let $n, d, \ell, t$ be positive integers such that $\ell = \frac{10r \log n}{\eps \lam}$, $t \le \frac{\lam d}{4r^2 \ell}$, $4000r^4 \log n < \eps^2 \lam^2 d$ and $\frac{\lam}{4r} \le \eps$. Suppose that $G$ is an $r$-line-graph on $n$ vertices, with $V(G) \subseteq A_1 \times \ldots \times A_r$, which is a $(\lam, d)$-expander, and let $x \in V(G)$. Then there is a set $Y \subseteq V(G)$ of size at least $(1 - 2\eps)n$ such that every $y \in Y$ can be reached from $x$ by a $\sigma$-path $P(y)$ of order at most $\ell$, and every $w \in (A_1 \cup \ldots \cup A_r) \setminus \coords(x)$ is in $\coords(P(y))$ for at most $\frac{n}{t}$ values of $y$. 
	\end{lem}

	\begin{proof}
		Write $A = A_1 \cup \ldots \cup A_r$ and $u = r\ell t$. So $u \le \frac{\lam d}{4r}$ and $\frac{u}{d} \le \eps$.

		Let $Y_0$ be a subset of $V(G)$ of maximum size for which there exists a collection of $\sigma$-paths $(P(y))_{y \in Y_0}$, such that $P(y)$ is a $\sigma$-path from $x$ to $y$ of order at most $\ell$ for $y \in Y_0$, and every $w \in A \setminus \coords(x)$ is in $\coords(P(y))$ for at most $\floor{\frac{n}{t}}$ values of $y$; fix such a collection $(P(y))_{y \in Y_0}$. Our task is to show that $|Y_0| \ge (1 - 2\eps)n$, so suppose otherwise.

		Let $F$ be the set of coordinates $w \in A \setminus \coords(x)$ such that $w \in \coords(P(y))$ for exactly $\frac{n}{t}$ values of $y \in Y_0$. By choice of $F$ and the upper bound on the order of $P(y)$, we have
		\begin{equation} \label{eqn:bad-coords}
			\frac{|F|n}{t} \le \sum_{y \in Y_0} |\coords(P(y))| \le r\ell n.
		\end{equation}
		It follows that $|F| \le r\ell t = u$. 

		Let $H$ be the graph obtained from $G$ by removing the vertices that meet the set $F$. By \Cref{lem:robust-expander}, $H$ is an $r$-line-graph on at least $(1 - \frac{u}{d})n \ge (1 - \eps)n$ vertices which is a $(\frac{\lam}{2}, \frac{d}{2})$-expander. Note that $x$ is in $H$ because $F$ is disjoint from $\coords(x)$. Thus, by \Cref{lem:expanding}, there is a subset $Y_1 \subseteq V(H)$, with $|Y_1| \ge (1 - \eps)|H| \ge (1 - \eps)^2n \ge (1 - 2\eps)n$, such that the vertices in $Y_1$ can be reached from $x$ by a $\sigma$-path in $H$ of order at most $\ell$ (here we use the inequality $500 r^4 \log |H| \le 500 r^4 \log n < \eps^2 \left( \frac{\lam}{2} \right)^2 \left( \frac{d}{2} \right)$). By assumption on the size of $Y_0$, there is a vertex $y \in Y_1 \setminus Y_0$. Let $P(y)$ be a $\sigma$-path in $H$ from $x$ to $y$ whose order is at most $\ell$; so $P(y)$ is a path in $G$ that avoids $F$. It follows that every $w \in A \setminus \coords(x)$ is in $\coords(P(y))$ for at most $\frac{n}{t}$ values of $y$ in $Y_0 \cup \{y\}$. This is a contradiction to the maximality of $Y_0$. Thus $|Y_0| \ge (1 - 2\eps)n$, as required.
	\end{proof}

	We now prove the main ingredient in our proof, namely that $r$-line-graphs which are dense expanders contain (short) $\sigma$-cycles.

	\begin{thm} \label{thm:cycles-in-expanders}
		Let $\sigma \in S_r$, let $\eps, \lam \in (0, 1)$, and let $n$ and $d$ be positive integers such that $d \ge \frac{4000r^5 (\log n)^2}{\eps^3 \lam^3}$, $\frac{\lam}{4r} \le \eps < \frac{1}{12}$ and $n$ is sufficienlty large. Let $G$ be an $r$-line-graph on $n$ vertices which is a $(\lam, d)$-expander. Then $G$ contains a $\sigma$-cycle of order at most $\frac{30r\log n}{\eps \lam}$. 
	\end{thm}

	\begin{proof}
		Let $A_1, \ldots, A_r$ be disjoint sets such that $V(G) \subseteq A_1 \times \ldots \times A_r$ and write $A = A_1 \cup \ldots \cup A_r$. Let $u = \frac{\lam d}{4r}$, write $\ell = \frac{10r \log n}{\eps \lam}$ and let $t = \frac{u}{r\ell}$. We claim that the following inequalities hold: $\frac{u}{d} \le \eps$ and $\frac{r \ell}{t} \le \eps$. The former is easy to check by the definition of $u$ and the lower bound on $\eps$. The latter is more tedious but follows directly from the choices of $u, \ell, t$ and the lower bound on $d$.

		For each vertex $x$ in $G$, let $Y(x) \subseteq V(G)$ be a set of size at least $(1 - 2\eps)n$ and let $P(x, y)$ be a $\sigma$-path of length at most $\ell$ in $G$ from $x$ to $y$, for $y \in Y(x)$, such that 
		\begin{equation} \label{eqn:good}
			\text{every $w \in A \setminus \coords(x)$ is in $\coords(P(x,y))$ for at most $\frac{n}{t}$ vertices $y$ in $Y(x)$, for $x \in V(G)$.}
		\end{equation}
		Such set $Y(x)$ and paths $P(x, y)$ exist by \Cref{lem:robust-expanding}.
		For each vertex $y$ in $G$, let $F(y)$ be the set of elements $w \in A \setminus \coords(y)$ that appear in more than $\frac{n}{t}$ sets $\coords(P(x, y))$ with $x \in V(G)$. Using a calculation as in \eqref{eqn:bad-coords}, it is easy to see that $|F(y)| \le u$ for every $y \in V(G)$. Let $G(y)$ be the graph obtained from $G$ by removing all vertices that meet $F(y)$. It follows from \Cref{lem:robust-expander} that $G(y)$ is an $r$-line-graph on at least $(1 - \frac{u}{d})n \ge (1 - \eps)n$ vertices, and it is also a $(\frac{\lam}{2}, \frac{d}{2})$-expander. By \Cref{lem:expanding}, there is a subset $X(y)$ of $V(G(y))$ with $|X(y)| \ge (1 - \eps)^2 n \ge (1 - 2\eps)n$, and $\sigma$-paths $Q(y, x)$ in $G(y)$ from $y$ to $x$ whose order is at most $\ell$, for $x \in X(y)$.

		Consider a vertex $x$ in $G$. Let $D(x)$ be a directed graph on vertices $V(G)$ where $yz$ is an edge if paths $P(x, y)$ and $P(y, z)$ are defined and $\coords(P(x,y)) \cap \coords(P(y, z)) = \coords(y)$; equivalently, $yz$ is an edge if the concatenation of $P(x, y)$ and $P(y, z)$ forms a $\sigma$-path in $G$ from $x$ to $z$. Given $y$ for which $P(x, y)$ is defined, the number of vertices $z$ for which $P(y, z)$ is defined but $yz$ is not an edge in $D(x)$ is at most $\frac{r\ell n}{t} \le \eps n$, by \eqref{eqn:good}. Since $P(y, z)$ is defined for at least $(1 - 2\eps)n$ vertices $z$, this implies that every vertex in $Y(x)$ has out-degree at least $(1 - 3\eps)n$ in $D(x)$. It follows that the number of edges in $D(x)$ is at least $(1 - 2\eps)n \cdot (1 - 3\eps) n \ge (1 - 5\eps)n^2$, and thus there are at least $(1 - 10\eps)n$ vertices in $G$ with in-degree at least $\frac{n}{2}$ in $D(x)$. 

		The previous paragraph implies that the number of pairs $(x, y)$ with $x, y \in V(G)$, such that $y$ has in-degree at least $\frac{n}{2}$ in $D(x)$, is at least $(1 - 10\eps)n^2$. Recall that the number of pairs $(x, y)$ with $x, y \in V(G)$, such that $Q(y, x)$ is defined, is at least $(1 - 2\eps)n^2$. It follows that there are at least $(1 - 12\eps)n^2$ pairs $(x, y)$ such that $y$ has in-degree at least $\frac{n}{2}$ in $D(x)$ and $Q(y, x)$ is defined. We claim that every such pair yields a $\sigma$-cycle in $G$ that passes through $x$ and $y$. 
			
		To see this, fix a pair $(x, y)$ as in the previous paragraph. Write $S = \coords(Q(y, x)) \setminus (\coords(x) \cup \coords(y))$. Then $|S| \le r\ell$, and $S$ is disjoint of $F(y)$, by choice of $Q(y, x)$. Let $Z$ be the in-neighbourhood of $y$ in $D(x)$; so $|Z| \ge \frac{n}{2}$. We claim that there is a vertex $z$ in $Z$ such that $P(x, z)$ and $P(z, y)$ both avoid $S$. To see this, first note that, by \eqref{eqn:good}, there are at most $\frac{r\ell n}{t} \le \eps n$ vertices $z$ in $Z$ such that $P(x, z)$ intersects $S$. Similarly, as $S$ is disjoint of $F(y)$ and by choice of $F(y)$,  there are at most $\frac{r \ell n}{t} \le \eps n$ vertices $z$ in $Z$ such that $P(z, y)$ meets $S$. It follows that there are at least $|Z| - 2\eps n \ge \frac{n}{4}$ vertices $z \in Z$ such that $\coords(P(x,y))$ and $\coords(P(y, z))$ are disjoint of $S$. Fix such $z$. The concatenation $P(x, z) P(z, y) Q(y, x)$ is a $\sigma$-cycle in $G$ (of order at most $3\ell$).
	\end{proof}

	Finally, we prove our main result, \Cref{thm:main}. It follows easily from the results above.

	\begin{proof}[Proof of \Cref{thm:main}]
		Let $\HH$ be an $r$-graph on $N$ vertices which does not contain a tight cycle. 
		By considering a random partition of $V(\HH)$ into $r$ parts, we can find an $r$-partite subgraph $\HH'$ of $\HH$ with at least $\frac{r!}{r^r} \cdot e(\HH)$ edges. 

		Write $e(\HH') = d N^{r-1}$, $n = e(\HH')$, $\lam = \frac{1}{2\log n}$ and $\eps = \frac{1}{20}$. Consider the $r$-line-graph $G$ that corresponds to $\HH'$. Then $\dens(G) = \frac{rn}{p(G)} \ge \frac{rdN^{r-1}}{rN^{r-1}} = d$ (see \eqref{eqn:density}). By \Cref{lem:existence-expanders}, there is a subgraph $G'$ of $G$ which is an $r$-line-graph and a $(\lam, \frac{d}{2r})$-expander; denote $m = |G'|$. By \Cref{thm:cycles-in-expanders}, we find that 
		\begin{equation*}
			\frac{d}{2r} < \frac{4000r^5 (\log m)^2}{\eps^3 \lam^3} 
			\le 2^8 \cdot 10^6 \cdot r^5 (\log n)^5.
		\end{equation*}
		Indeed, otherwise \Cref{thm:cycles-in-expanders} yields a $\sigma$-cycle in $G'$ (of length at most $r \cdot \frac{30r\log m}{\eps \lam} \le 1200r^3 (\log n)^2$), which corresponds to a tight cycle in $\HH'$, contradicting the assumption on $\HH$.
		It follows that $d \le 10^9 r^6 (\log n)^5 \le 10^9 r^{11} (\log N)^5$ (using $n \le N^r$), implying that
		\begin{equation*}
			e(\HH) 
			\le \frac{r^r}{r!} \cdot e(\HH') 
			\le \frac{10^9 r^{r+11}}{r!} \cdot N^{r-1} (\log N)^5 
			= O(N^{r-1} (\log N)^5),
		\end{equation*}
		as required.
	\end{proof}

\section{Conclusion} \label{sec:conclusion}
	We proved that the maximum possible number of edges in an $n$-vertex $r$-graph with no tight cycles is at most $O(n^{r-1} (\log n)^5)$, thus pinning down this extremal number up to a polylogarithmic factor. Specifically, we showed that every $r$-line-graph $G$ which is a $(\lam, d)$-expander, with $d$ sufficiently large, contains a $\sigma$-cycle. In fact, our proof implies that there is a $\sigma$-cycle between almost every two vertices in $G$. However, it is not clear if the same should hold for every two vertices in $G$ whose coordinate sets are disjoint. Even the following, slightly weaker question, remains open: in an $r$-line-graph which is a dense expander, can every two vertices which do not share coordinates be joined by a $\sigma$-path?

	It is natural to consider a similar question to the one discussed in this paper, where instead of forbidding all tight cycles, we forbid a tight cycle of given length $\ell$. This was addressed for $\ell$ which is linear in $n$ by Allen, B\"ottcher, Cooley and Mycroft \cite{allen2017tight}, and an unpublished result of Verstra\"ete considered the case $\ell = 24$ and $r = 3$. When $\ell$ is not divisible by $r$, there exist $n$-vertex $r$-graphs with $\Omega(n^r)$ edges and no tight cycles of length $\ell$; indeed, any dense $r$-partite $r$-graph would do. Conlon (see \cite{mubayi2011hypergraph}) asked the following question for fixed $\ell$ which is divisible by $r$.
	\begin{qn}[Conlon]
		Given $r \ge 3$, is there $c = c(r)$ such that whenever $\ell > r$ and $\ell$ is divisible by $r$, every $n$-vertex $r$-graph with no tight cycle of length $\ell$ has at most $O(n^{r-1+c/\ell})$ edges?
	\end{qn}

	We note that a lot more is known about the number of edges in an $r$-graph with no Berge or loose cycle of given lengths; see, e.g., \cite{furedi2014hypergraph,gyori2012hypergraphs,kostochka2015turan,frankl1987exact,gyori20123,jiang2018cycles,jiang2020linear}.

	\subsection*{Acknowledgements}

		I am grateful to Liana Yepremyan for introducing this problem to me.

\bibliography{tight-cycle}
\bibliographystyle{amsplain}
	
\end{document}